\documentclass[11pt,reqno]{amsart}
\usepackage[letterpaper,textwidth=6.18in,textheight=8.92in,centering]{geometry}
\usepackage[T1]{fontenc}
\usepackage{lmodern}
\usepackage{amsmath,amsthm,amssymb}
\usepackage{tikz-cd}
\usepackage[hidelinks,unicode]{hyperref}

\theoremstyle{plain}
\newtheorem{theorem}{Theorem}[section]
\newtheorem{lemma}[theorem]{Lemma}
\newtheorem{proposition}[theorem]{Proposition}
\newtheorem{corollary}[theorem]{Corollary}
\newtheorem*{maintheorem}{Main Theorem}
\theoremstyle{definition}
\newtheorem{definition}[theorem]{Definition}
\theoremstyle{remark}
\newtheorem{remark}[theorem]{Remark}
\numberwithin{equation}{section}
\DeclareMathOperator{\fd}{fd}
\DeclareMathOperator{\Mod}{Mod}
\DeclareMathOperator{\Spec}{Spec}
\DeclareMathOperator{\Hom}{Hom}
\DeclareMathOperator{\gr}{gr}
\DeclareMathOperator{\adic}{adic}
\newcommand{\dtimes}{\otimes^{\mathbb L}}
\newcommand{\LSym}{\mathbb L\operatorname{Sym}}
\hypersetup{
  pdftitle={Quillen's finiteness conjecture for the cotangent complex in
  positive characteristic},
  pdfauthor={Timo Schürg}
}
\begin{document}
\title[Quillen's finiteness conjecture]{Quillen's finiteness conjecture for the
cotangent complex in positive characteristic}
\author{Timo Sch\"urg}
\address{Department of Mathematics and Natural Sciences, Darmstadt University of Applied Sciences, Germany}

\begin{abstract}
Let $A\to B$ be a homomorphism of Noetherian rings, with $B$ of
positive characteristic. We prove that if the relative cotangent complex
$L_{B/A}$ has finite flat dimension over $B$, then its flat dimension
is at most two. The proof combines a generalization of Turner's
nilpotence criterion for Andr\'e--Quillen homology with a filtration
on a derived difference algebra. Over a residue field of characteristic
$p$, each suspended homotopy operation of weight $p$ multiplies the
filtration index by $p$, while the bounded positive homotopy of a
derived quotient gives a uniform bound on the filtration length.
The argument also applies to maps $A\to B$ of simplicial rings with
finite Noetherian homotopy, provided that $\pi_0A\to\pi_0B$ is
surjective and $\pi_0B$ has positive characteristic.
\end{abstract}
\maketitle
\raggedbottom

\section{Introduction}\label{sec:intro}

In \cite{quillen-proc}, Quillen formulated two conjectures on the
finiteness properties of the cotangent complex. Let $f\colon A\to B$
be an essentially finite type homomorphism of Noetherian rings, and
write $L_{B/A}$ for its relative cotangent complex.

\medskip
\noindent\textbf{Conjecture I.}
If $L_{B/A}$ has finite flat dimension over $B$ and $B$ has finite
flat dimension over $A$, then $f$ is locally complete intersection.
In particular, $\fd_B L_{B/A}\leq1$.

\medskip
\noindent\textbf{Conjecture II.}
If $L_{B/A}$ has finite flat dimension over $B$, then
$\fd_B L_{B/A}\leq2$.

\medskip
Avramov generalized and proved the first conjecture \cite{avramov}.
Turner gave a different proof in positive characteristic
\cite{turner-inventiones}, and Briggs and Iyengar subsequently
established stronger rigidity results \cite{briggs-iyengar}.
For an introduction to the subject, see \cite{iyengarintro}.

Known cases of the second conjecture include complete intersection
base rings \cite{avramov} and algebra retracts
\cite{avramoviyengar}. In characteristic two, Turner proved the
conjecture for Cohen--Macaulay base rings by a different method
\cite[Theorem A]{turner}. We extend his approach to every positive
residue characteristic and prove the following result
(Theorem~\ref{thm:local-quillen}).

\begin{maintheorem}
Let $A\twoheadrightarrow B$ be a surjective homomorphism of
Noetherian local rings whose common residue field has positive
characteristic. Then
\[
 \fd_B L_{B/A}<\infty\quad\Longrightarrow\quad\fd_B L_{B/A}\leq2.
\]
\end{maintheorem}

Conjecture II then follows for morphisms with target of positive characteristic
using Cohen factorizations (Corollary~\ref{cor:quillen-finite-characteristic}).
We also obtain a similar result for local homomorphisms with positive residue
characteristic: if the relative cotangent complex is  perfect, then it has
Tor-amplitude contained in $[0,2]$
(Corollary~\ref{cor:quillen-mixed-characteristic}). This includes mixed
characteristic.

We briefly summarize the proof strategy.
Avramov's and Turner's proofs share two steps: associate a derived
algebra to a local homomorphism, then apply a criterion that detects
the required vanishing. For a local surjection $A\to A/I$ with
residue field $k$, both approaches use the derived tensor product
$F=k\dtimes_A(A/I)$, given by the homotopy pushout
\[
\begin{tikzcd}[column sep=large, row sep=large]
  A
    \arrow[r]
    \arrow[d]
  & k \arrow[d] \\
  A/I \arrow[r]
  & F
\end{tikzcd}
\]
Cotangent base change gives
\[
 k\dtimes_{A/I}L_{(A/I)/A}\simeq k\dtimes_F L_{F/k}.
\]
Thus $F$ retains the Andr\'e--Quillen homology of the original
homomorphism with residue-field coefficients.

Avramov's proof of Conjecture I studies the deviations of a
differential graded model for $F$. His criterion characterizes
complete intersection homomorphisms by subexponential growth of
these deviations. Under the hypotheses of Conjecture I, he obtains
the required growth bound from eventual vanishing of
Andr\'e--Quillen homology.

Turner works with $F$ as a simplicial commutative algebra and studies
the higher divided power operations on its homotopy. His criterion
for Conjecture II in characteristic two relates vanishing of
Andr\'e--Quillen homology above degree two to nilpotence of suitable
strings of these operations.

Our proof develops this strategy in three steps. First, we establish
a version of Turner's nilpotence criterion in every positive
characteristic,
using the description of homotopy operations in Fu and Nuiten
\cite{FN}. The adic filtration supplies the detection
argument needed for this extension. This material is presented in
Sections~\ref{sec:suspended} and~\ref{sec:criterion}. The resulting
criterion is Corollary~\ref{thm:nilpotence-criterion}.

Second, we replace $F$ by a different derived algebra. Choose
generators $f_1,\ldots,f_c$ of $I$ and form the derived quotient
$R=A//(f_1,\ldots,f_c)$. The map $A\to R$ is quasi-smooth, the
derived analogue of a complete intersection: its cotangent complex
is $L_{R/A}\simeq R^c[1]$. Moreover, $\pi_0R=A/I$, so the natural
map $R\to A/I$ gives a factorization $A\to R\to A/I$.

Put $F'=k\dtimes_A R$. Together with the residue field $k$, the
factorization gives the following diagram, in which the three
indicated squares are homotopy pushouts:
\begin{equation}\label{eq:double}
\begin{tikzcd}[column sep=large, row sep=large]
  A
    \arrow[r]
    \arrow[d]
  & k \arrow[d] \\
  A//(f_1,\dots,f_c)
    \arrow[r]
    \arrow[d]
  & F'
    \arrow[r]
    \arrow[d]
  & k \arrow[d] \\
  A/I \arrow[r]
  & F \arrow[r]
  & D
\end{tikzcd}
\end{equation}

We call $D=k\dtimes_R(A/I)$ the \emph{difference algebra}.
It is $1$-connected, and the Jacobi--Zariski sequence identifies its
Andr\'e--Quillen homology with that of $A\to A/I$ in degrees at
least three. Thus the vanishing
required by Conjecture II can be checked on $D$.

Third, we construct the \emph{boundary filtration} $G^\bullet$ on
$\pi_mD$, for $m\geq2$. It is a decreasing filtration with
$G^1\pi_mD=\pi_mD$ and $G^q\pi_mD=0$ for all $q$ greater than a fixed bound $c$. Writing $\nu_a$ for a suspended homotopy
operation in characteristic $p$, the filtration has the special property that
\[
 \nu_a(G^q\pi_mD)\subseteq G^{pq}\pi_{m+2a(p-1)}D
 \quad(m>2a),
\]
Thus every suspended operation
multiplies the filtration index by $p$, and a string of $r$ such
operations vanishes as soon as $p^r>c$.

The bound $c$ comes from the fiber $J=\operatorname{fib}(R\to A/I)$,
whose homotopy is concentrated in degrees $1,\ldots,c$. Its derived
adic powers $J^q$ are $q$-connective, so their maps to $J$ are null
for $q>c$. These nullhomotopies persist after base change to $k$
and give the bound on $G^\bullet$, even when $D$ has unbounded
homotopy. Section~\ref{sec:quotient} constructs the filtration from
the relative adic filtration of the unit $k\to D$ and proves its
compatibility with suspended operations.

Geometrically, the boundary filtration records extension to derived
infinitesimal neighborhoods. This perspective is motivated by their use in pro-descent,
as in the work of Kerz, Strunk and Tamme \cite{tamme} and Kelly,
Saito and Tamme \cite{kst}; see Remark~\ref{rem:geometric-boundary}
for the construction.

Applying the uniform nilpotence criterion to $D$ and comparing its
Andr\'e--Quillen homology with that of $A\to A/I$ completes the local
argument. Cohen factorization then gives
Corollaries~\ref{cor:quillen-finite-characteristic}
and~\ref{cor:quillen-mixed-characteristic}.
Section~\ref{sec:non-discrete} adapts the
bounded-fiber construction to maps of simplicial rings.

\subsection{Notation and conventions}\label{subsec:notation}

Throughout the paper, $k$ is a field of characteristic $p>0$.
All rings and algebras are commutative, and we use simplicial rings
as models for derived rings. A \emph{simplicial $k$-algebra} means
a simplicial commutative $k$-algebra. Derived tensor products are
denoted $\dtimes$.

Grading is homological, and $k[n]$ denotes $k$ placed in degree $n$.
For a simplicial commutative ring $P$, we write $\Mod_P$ for the
stable $\infty$-category of $P$-modules with its standard $t$-structure.
A $P$-module $M$ is \emph{$n$-connective} if
$M\in(\Mod_P)_{\geq n}$, equivalently $\pi_iM=0$ for $i<n$. If
$M\in(\Mod_P)_{\geq q}$ and $N\in(\Mod_P)_{\leq b}$ with $q>b$, every
map $M\to N$ is nullhomotopic.

A simplicial ring $P$ is \emph{coherent} if $\pi_0P$ is coherent
and each $\pi_iP$ is finitely presented over $\pi_0P$. A $P$-module
is \emph{almost perfect} if, for every integer $n$, it admits a map
from a perfect $P$-module whose cofiber is $n$-connective. Over a
coherent $P$, this is equivalent to having bounded-below, finitely
presented homotopy groups \cite[Appendix~A]{BMN}. In particular, an
almost perfect $k$-module has finite-dimensional homotopy in each
degree and is bounded below; its homotopy may be unbounded above.
The stronger condition of finite Noetherian homotopy is defined
in Section~\ref{sec:non-discrete}.

For a map $A\to B$ of simplicial commutative rings, $L_{B/A}$ denotes
the relative cotangent complex. Its Andr\'e--Quillen homology with
coefficients in a $B$-module $M$ is $\pi_i(M\dtimes_B L_{B/A})$
\cite{andre,quillen-proc}. For a connective $P$-module $M$, the bound
$\fd_P M\leq n$ means that $\pi_i(N\dtimes_P M)=0$ for all $i>n$
and every discrete $P$-module $N$.

If $\pi_0P=k$ is a field, then for every connective $P$-module $M$,
\begin{equation}\label{eq:flat-dimension-field}
 \fd_P M
 =\sup\bigl\{i:\pi_i(k\dtimes_P M)\ne0\bigr\},
\end{equation}
since every discrete $P$-module is a $k$-vector space. The same
equality holds when $P$ is a discrete Noetherian local ring with
residue field $k$ and $M$ is almost perfect
\cite[Tag~068V]{stacks-project}.

An augmented simplicial $k$-algebra $X$ is \emph{connected} if
$\pi_0X=k$ and \emph{$1$-connected} if also $\pi_1X=0$.

We write $\LSym_k(M)$ for the free simplicial $k$-algebra on a
connective $k$-module $M$, and $\LSym_k^a(M)$ for its weight-$a$
summand, the $a$-th derived symmetric power. In particular,
$S(d)=\LSym_k(k[d])$ is the free augmented algebra on one generator
$\iota_d$ in degree $d$. We similarly write $\mathbb L\Lambda_P^q$
and $\mathbb L\Gamma_P^q$ for derived exterior and divided powers
over a simplicial ring $P$.

By a filtered simplicial commutative algebra we mean a decreasing
multiplicative filtration in the derived sense of
\cite[Definition~2.2]{BMN}, including compatible symmetric-power
maps. Its positive stages are modules equipped with
the corresponding products. For a decreasing filtration $K^\bullet$,
the notation $K^i/K^j$, $j\geq i$, denotes the cofiber in modules,
and $\gr^iK=K^i/K^{i+1}$. We call the filtration \emph{complete} if
$\operatorname*{holim}_{q\geq1}K^q\simeq0$.

Adic filtrations and ideal powers are always understood in the
derived sense. For a map $P\to Q$ surjective on $\pi_0$, write
$I^\bullet=\adic(P\to Q)$. Its basic properties are
\[
 I^0=P,\qquad I^1\simeq\operatorname{fib}(P\to Q),\qquad
 \gr^0I\simeq Q,
\]
and, for $q\geq1$,
\[
 \gr^qI\simeq\LSym_Q^q(L_{Q/P}[-1])
\]
\cite[Definition~2.5 and Lemma~2.12]{BMN}. We will apply this construction
both to augmentations $X\to k$ and to unit maps $k\to D$; their
filtrations have different roles in the proof.

\subsection*{Related work}

The principal contribution of this paper is the method based on the difference
algebra and its boundary filtration, both introduced in the earlier
characteristic-two version. Together, these constructions turn boundedness of a
presentation fiber into uniform nilpotence of suspended homotopy operations,
even when the difference algebra itself has unbounded homotopy.

Liu \cite{Liu} subsequently extended the homotopy-operation argument to
arbitrary positive characteristic, establishing a version of Turner's nilpotence
criterion and applying it within the difference-algebra approach.
Sections~\ref{sec:suspended} and~\ref{sec:criterion} give an independently
developed treatment of this extension, using relative lifting in the adic
filtration to obtain both the criterion and the required compatibility of
operations with the boundary filtration.

\subsection*{Acknowledgements}

I thank Herman Stel for reading several early versions of this paper and for his
valuable feedback, and James Turner for encouraging me to pursue the extension
to simplicial rings he conjectured.  I also thank Christian Dahlhausen and Can
Yaylali for helpful discussions in the early stages of the project, Lukas
Brantner for explaining his approach to Koszul duality, Jeroen Hekking for
pointing me to his joint paper \cite{hekking} on derived ideals, and Jiaqi Fu
for explaining his work on homotopy operations.

\subsection*{Declaration on the use of artificial intelligence}

In the world of mountaineering, there is about as much discussion about the style
in which a peak was climbed, as the actual climb.  If this paper was a climb, it
could not claim impeccable style, but would have used supplemental oxygen
and occasional fixed ropes. My key contributions from the pre-AI era were
using the difference algebra, and putting the boundary filtration on this algebra.
I was unable to obtain the bound where the filtration vanishes and gave up on
the approach.

I explored other approaches to Quillen's conjecture using Koszul duality with
ChatGPT.  During one of these exchanges, ChatGPT mentioned the bound used in
this paper, allowing me to finish the original approach.

I take full responsibility for the contents of this paper and any remaining errors.

\section{Suspended operations and the adic filtration}\label{sec:suspended}

The goal of this section is to establish an analogue in arbitrary
positive characteristic of Turner's early survival argument in his
analysis of the Quillen spectral sequence
\cite{turner}.
We express survival through the early pages by lifts to quotients of
the adic filtration. For a quotient $K^q/K^{q+r}$, we call $r$ its
\emph{gap}. Proposition~\ref{prop:iterated-gap} shows that each
iteration multiplies both the filtration index and the gap by $p$.

Throughout this section, $X$ is an augmented $1$-connected simplicial
$k$-algebra, and each $\pi_iX$ is finite-dimensional over $k$.

Compared to Turner's arguments in \cite{turner}, 1-connectedness allows us to
save some effort,  as we do not have to deal with any Adem-relations.

The augmentation $X\to k$ defines the adic filtration
$F^\bullet X=\adic(X\to k)$, with $F^0X=X$ and
$F^1X=\operatorname{fib}(X\to k)$, the derived augmentation ideal. By
\cite[Construction~5.36(d) and Propositions~4.26,~5.40]{BM},
the associated graded is free on the cotangent fiber:
\[
 \gr^1 X\simeq k\dtimes_X L_{X/k},
 \qquad \gr^q X\simeq\LSym_k^q(\gr^1X)\quad(q\ge1).
\]
Here $\gr^q X$ abbreviates $\gr^q(F^\bullet X)$.
For general filtered simplicial commutative $k$-algebras $K^\bullet$,
we use the framework of \cite[Definition~2.2]{BMN}, allowing arbitrary
degree-zero associated graded algebra $\gr^0K$. The algebra structure
includes power maps
\[
 P_s^{(b)}:\LSym_k^b(K^s)\longrightarrow K^{bs}
 \qquad(s\ge0,\ b\ge1).
\]
They commute with the transition maps. Since the associated-graded
functor commutes with the derived symmetric algebra monad
\cite[Remark~2.4]{BMN}, they induce the corresponding power maps
on the associated graded, denoted by
\[
 \overline P_q^{(b)}:\LSym_k^b(\gr^qK)\longrightarrow\gr^{bq}K
 \qquad(q\ge0,\ b\ge1).
\]

A class in $\pi_N\LSym_k^b(k[d])$ represents an operation of
\emph{weight} $b$ on degree-$d$ homotopy. Given a class
of degree $d$ in an algebra, one evaluates the universal class along
the corresponding map from $S(d)$. In a graded algebra an input of
weight $q$ gives an output of weight $bq$. In particular, when we
apply an operation to a class in $\pi_d\gr^qK$, we evaluate it in
the whole associated graded algebra $\gr^\bullet K$.
For $T=\LSym_k^b$, the suspension map $\Sigma T(V)\to T(\Sigma V)$
is induced by applying $T$ to the suspension diagram $0\leftarrow V\to0$.
We write $\sigma\beta$ for the image of $\Sigma\beta$ under this map.

We use the operations of \cite[Proposition~2.12 and Notation~2.14]{FN},
obtained over $k$ by base change of their universal classes from
$\mathbb F_p$. For $a\ge1$ and $m\ge2a$, the operation
\[
 \nu_a:\pi_mX\longrightarrow\pi_{m+2a(p-1)}X
\]
has weight $p$. For $p=2$, it is the operation $\delta_{2a}$ in the
usual notation \cite[Remark~1.5]{FN}.

The key property for our argument is that these operations are
represented by suspensions in the strict range $m>2a$. Writing
$\nu_a^{(m)}$ for the universal class on degree $m$, we have
\[
 \nu_a^{(m)}=\sigma\nu_a^{(m-1)}
 \qquad(m>2a).
\]
We call this the \emph{suspended range}. The suspension identity lets
us apply the operation one degree lower to the boundary of a relative
class, then use the cofiber sequence to construct a lift of its value.

We single out the words that will be used throughout the argument.

\begin{definition}\label{def:iterated-operations}
For $s\ge0$, define the words $b_{p,s}$ recursively by
\[
 b_{p,0}=\mathrm{id},
 \qquad b_{p,s+1}=\nu_{p^s}\circ b_{p,s}.
\]
Thus $b_{p,s}=\nu_{p^{s-1}}\circ\cdots\circ\nu_{p^0}$ for $s\ge1$.
On an input of degree $d\ge3$, the word $b_{p,s}$ has weight $p^s$
and output degree
\[
 t_s=d+2p^s-2.
\]
\end{definition}

\begin{remark}
Every factor is applied in the suspended range, since the input to
$\nu_{p^j}$ has degree $t_j$ and $t_j-2p^j=d-2>0$.
This difference is independent of $j$ and explains the threshold
$d\ge3$ in the detection argument.
\end{remark}

We first record the compatibility of filtered power maps with
relative quotients and the associated graded.

Let $K^\bullet$ be a nonnegatively indexed filtered simplicial commutative
$k$-algebra, and let $q\ge0$ and $r,b\ge1$. Put $T=\LSym_k^b$ and
$C_b=\operatorname{cofib}\bigl(T(K^{q+r})\to T(K^q)\bigr)$, and let
\[
 \pi:K^q/K^{q+r}\longrightarrow\gr^qK,
 \qquad
 \rho:K^{bq}/K^{b(q+r)}\longrightarrow\gr^{bq}K
\]
be the quotient maps. The filtered power maps induce a map on cofibers
\[
 m:C_b\longrightarrow K^{bq}/K^{b(q+r)}.
\]
Applying $T$ to the canonical nullhomotopy of
$K^{q+r}\to K^q\to K^q/K^{q+r}$ and using the universal property
of $C_b$ gives a natural comparison map
\[
 c:C_b\longrightarrow T(K^q/K^{q+r}),
\]
since $T(0)=0$.

\begin{lemma}\label{lem:quotient-compatibility}
With this notation, the square
\[
\begin{tikzcd}[column sep=large, row sep=large]
 C_b \arrow[r, "m"] \arrow[d, "T(\pi)c"']
   & K^{bq}/K^{b(q+r)} \arrow[d, "\rho"] \\
 T(\gr^qK) \arrow[r, "\overline P_q^{(b)}"']
   & \gr^{bq}K
\end{tikzcd}
\]
commutes naturally in $K^\bullet$.
\end{lemma}

\begin{proof}

All functors appearing in this square preserve sifted
colimits. The connective free-algebra extension machinery of \cite[\S2.1,
Remark~A.12 and Proposition~A.16]{BMN} therefore reduces the comparison to
polynomial algebras on finitely many generators with prescribed nonnegative
filtration weights, where $K^s$ is spanned by monomials of weight at least $s$.

For these algebras, the cofibers are ordinary quotients. Both composites send the
class of a symmetric product of elements of \(K^q\) to the class of their
product in \(K^{bq}/K^{bq+1}\). This equality is natural under
filtration-preserving algebra maps and therefore extends.

\end{proof}

To control the words above, we now prove a lifting statement for any
operation represented by a suspension. We will use it at positive
filtration index to construct the relative lifts needed for detection,
and at index zero to control the boundary filtration of
Section~\ref{sec:quotient}.

\begin{lemma}\label{lem:relative}
Let $K^\bullet$ be a nonnegatively indexed filtered simplicial
commutative $k$-algebra. Fix $q\ge0$, $r,b\ge1$ and $d,N\ge2$.
Suppose the operation $\alpha$ is represented by $\sigma\beta$, where
$\beta\in\pi_{N-1}\LSym_k^b(k[d-1])$.
If $x\in\pi_d\gr^q K$ lifts to $\pi_d(K^q/K^{q+r})$, then its value
$\alpha(x)\in\pi_N\gr^{bq}K$ lifts to
\[
 \pi_N(K^{bq}/K^{b(q+r)}).
\]
\end{lemma}

\begin{proof}
Put $T=\LSym_k^b$, and use $C_b$, $\pi$, $\rho$, $m$ and $c$
as in Lemma~\ref{lem:quotient-compatibility}. Represent the given
classes by maps
\[
 x:k[d]\longrightarrow\gr^qK,
 \qquad
 \beta:k[N-1]\longrightarrow T(k[d-1]).
\]
The suspended class $\sigma\beta$ is represented by
\[
 k[N]\xrightarrow{\Sigma\beta}\Sigma T(k[d-1])
 \xrightarrow{\sigma}T(k[d]),
\]
and its value $\alpha(x)$ is represented by the composite
\[
 k[N]\xrightarrow{\Sigma\beta}\Sigma T(k[d-1])
 \xrightarrow{\sigma}T(k[d])
 \xrightarrow{T(x)}T(\gr^qK)
 \xrightarrow{\overline P_q^{(b)}}\gr^{bq}K.
\]
Choose a lift $\bar x:k[d]\to K^q/K^{q+r}$, together with a
homotopy $\pi\bar x\simeq x$. By stability, this determines a
map of cofiber sequences
\[
\begin{tikzcd}[column sep=large, row sep=large]
 k[d-1] \arrow[r] \arrow[d, "u"']
   & 0 \arrow[r] \arrow[d]
   & k[d] \arrow[d, "\bar x"] \\
 K^{q+r} \arrow[r]
   & K^q \arrow[r]
   & K^q/K^{q+r}.
\end{tikzcd}
\]
The left square includes the nullhomotopy of
$k[d-1]\to K^{q+r}\to K^q$ specified by $\bar x$.
Apply $T$ to this square, form cofibers, and use the filtered power
maps. Since $T(0)=0$, this gives a diagram whose rows are cofiber
sequences:
\[
\begin{tikzcd}[column sep=2em, row sep=large]
 T(k[d-1]) \arrow[r] \arrow[d, "T(u)"']
   & 0 \arrow[r] \arrow[d]
   & \Sigma T(k[d-1]) \arrow[d, "g"] \\
 T(K^{q+r}) \arrow[r] \arrow[d, "P_{q+r}^{(b)}"']
   & T(K^q) \arrow[r] \arrow[d, "P_q^{(b)}"]
   & C_b \arrow[d, "m"] \\
 K^{b(q+r)} \arrow[r]
   & K^{bq} \arrow[r]
   & K^{bq}/K^{b(q+r)}.
\end{tikzcd}
\]
Define the candidate lift by
\[
 \bar z \colon
 k[N]\xrightarrow{\Sigma\beta}\Sigma T(k[d-1])
 \xrightarrow{g}C_b\xrightarrow{m}K^{bq}/K^{b(q+r)}.
\]

Naturality of the cofiber comparison gives
$cg\simeq T(\bar x)\sigma$. Together with $\pi\bar x\simeq x$,
this proves commutativity of the left square below. The right
square commutes by Lemma~\ref{lem:quotient-compatibility}:
\[
\begin{tikzcd}[column sep=large, row sep=large]
 \Sigma T(k[d-1]) \arrow[r, "g"] \arrow[d, "\sigma"']
   & C_b \arrow[r, "m"] \arrow[d, "T(\pi)c"]
   & K^{bq}/K^{b(q+r)} \arrow[d, "\rho"] \\
 T(k[d]) \arrow[r, "T(x)"']
   & T(\gr^qK) \arrow[r, "\overline P_q^{(b)}"']
   & \gr^{bq}K.
\end{tikzcd}
\]
Precomposing its upper row with
$\Sigma\beta:k[N]\to\Sigma T(k[d-1])$ gives
\[
\begin{aligned}
 \rho\bar z
 &=\rho m g\,\Sigma\beta\\
 &\simeq\overline P_q^{(b)}T(\pi)c g\,\Sigma\beta\\
 &\simeq\overline P_q^{(b)}T(x)\sigma\,\Sigma\beta.
\end{aligned}
\]
The last composite represents $\alpha(x)$, so
$\rho_*(\bar z)=\alpha(x)$.
\end{proof}

Iterating the lifting lemma along the word $b_{p,s}$ multiplies both
the filtration index and the gap by $p$ at each step. The resulting
relative lifts express early survival in Quillen's spectral sequence
in terms of the adic filtration.

\begin{proposition}\label{prop:iterated-gap}
For every $y\in\pi_d\gr^1X$ with $d\ge3$, the class $b_{p,s}(y)$ in
$\pi_{t_s}\gr^{p^s}X$ admits a lift
\begin{equation}\label{eq:relative-iterate}
 \bar z_s\in\pi_{t_s}(F^{p^s}X/F^{2p^s}X)
 \qquad(s\ge0).
\end{equation}
\end{proposition}

\begin{proof}
Start with $\bar z_0=y$. At the $j$-th step, $t_j-2p^j=d-2>0$,
so $\nu_{p^j}$ is suspended. Apply Lemma~\ref{lem:relative} with
$q=r=p^j$ and $b=p$. The new degree is $t_j+2p^j(p-1)=t_{j+1}$.
\end{proof}

\section{A generalized Turner criterion}\label{sec:criterion}

We now combine these relative lifts with detection in a free algebra
to obtain a uniform nilpotence criterion. Following Turner's strategy,
if the top degree of the cotangent fiber of $X$ is $d\ge3$, we construct
an augmented map $f:X\to S(d)$ and find a class in $\pi_*X$ whose
image supports nonzero composites of suspended operations of every
length.

As in the previous section, $X$ is an augmented $1$-connected simplicial
$k$-algebra, and each $\pi_iX$ is finite-dimensional over $k$.

We first record the nonvanishing of the chosen words in a free algebra.
\begin{lemma}\label{lem:free-words}
For $d\ge3$ and $s\ge1$, the free generator satisfies
\[
 b_{p,s}(\iota_d)\ne0
 \quad\text{in}\quad \pi_{d+2p^s-2}\LSym_k^{p^s}(k[d]).
\]
\end{lemma}

\begin{proof}
Over $\mathbb F_p$, reading the operations from the outermost one
inwards gives the sequence
\[
 \bigl(2p^{s-1}(p-1),\,\ldots,\,2(p-1);\,\iota_d\bigr)
\]
in the notation of \cite[Definition~2.37]{FN}. Write $i_1,\ldots,i_s$
for its degree increments. They are greater than one and divisible
by $2(p-1)$, and satisfy $i_j=p\,i_{j+1}$ for $j<s$.
The input to $\nu_{p^j}$ has degree
$t_j>2p^j$, which verifies the remaining admissibility inequalities.
Thus the sequence is $1$-admissible and gives a basis element by
\cite[Theorem~2.39]{FN}. Its weight is $p^s$, and extension of
scalars to $k$ preserves its nonvanishing.
\end{proof}

We also need the connectivity of the higher weights in the free target.
The adic filtration of $S(d)$ splits by weight. For $d\ge2$,
applying \cite[Proposition~3.1.3]{Lurie} to each summand gives
\begin{equation}\label{eq:sphere-connectivity}
 F^qS(d)\simeq\bigoplus_{a\ge q}\LSym_k^a(k[d]),
 \qquad \pi_i F^qS(d)=0\quad(i<d+2q-2,\ q\ge1).
\end{equation}
In particular, the quotient by weights at least two identifies
$\tau_{\le d}S(d)$ with the square-zero algebra $k\oplus k[d]$.

The cotangent bound allows a top-degree functional to lift through
the Postnikov tower of $S(d)$.

\begin{lemma}\label{lem:detecting-map}
Suppose $\fd_X L_{X/k}\le d$ and $d\ge3$. For every nonzero
$y\in\pi_d\gr^1X$, there is an augmented map $f:X\to S(d)$ such that
$(\gr^1 f)_*(y)=\iota_d$, where $\iota_d$ also denotes the generator
of $\pi_d\gr^1 S(d)$.
\end{lemma}

\begin{proof}
Since $k$ is a field, the universal property of the cotangent complex
allows us to choose an augmented map
\[
 \chi:X\longrightarrow k\oplus k[d]\simeq\tau_{\le d}S(d)
\]
whose corresponding functional on $\pi_d\gr^1X$ takes $y$ to $1$.

Suppose $\chi$ has been lifted to $\tau_{\le m}S(d)$ for $m\ge d$.
The next Postnikov map is a square-zero extension
\cite[Lemma~2.2.1.1]{HAGII}. Write $V=\pi_{m+1}S(d)$; its
$X$-module structure factors through $X\to\pi_0X=k$. The
obstruction to extending the lift belongs to
\[
 [L_{X/k},V[m+2]]_X
 \cong [\gr^1X,V[m+2]]_k
 \cong\Hom_k(\pi_{m+2}\gr^1X,V)=0 .
\]
The vanishing follows from $\fd_X L_{X/k}\le d$ and
\eqref{eq:flat-dimension-field}, since $m+2>d$.
Successive compatible lifts give $f$ since $S(d)$ is the limit of its
Postnikov tower. The quotient $S(d)\to k\oplus k[d]$ is the universal
derivation on the free generator after augmentation, so the functional
defined by $\chi$ is precisely the one induced by $\gr^1f$.
\end{proof}

The map $f$ detects the chosen class on cotangent fibers. To detect
a class in ordinary homotopy, we must turn the relative lifts of
Proposition~\ref{prop:iterated-gap} into actual classes of $\pi_*X$.
The next two lemmas show that the adic stages become sufficiently
connective to remove the obstruction to this further lifting.

\begin{lemma}\label{lem:complete-filtration}
Let $M^{q+1}\to M^q$ be a filtration of modules over a simplicial
commutative ring with $\operatorname*{holim}_{q\ge1}M^q=0$.
If $\gr^q M\in\Mod_{\ge cq}$ for all $q\ge1$ and a fixed integer
$c\ge1$, then $M^q\in\Mod_{\ge cq}$.
\end{lemma}

\begin{proof}
Fix $i<cq$. The transition maps on $\pi_iM^j$ are surjective for
$j\ge q$, by the cofiber sequence defining $\gr^jM$. Therefore
$\varprojlim_{j\ge q}\pi_i M^j\to\pi_iM^q$ is surjective. Its
source vanishes by the Milnor exact sequence and completeness.
\end{proof}

\begin{lemma}\label{lem:connectivity}
As above, assume $X$ is $1$-connected and each $\pi_iX$ is
finite-dimensional over $k$. Then
\begin{equation}\label{eq:adic-connectivity}
 F^qX\in\Mod_{k,\ge2q}\qquad(q\ge1).
\end{equation}
\end{lemma}

\begin{proof}
Since $\pi_0X=k$ is complete local and each $\pi_iX$ is finite over
$k$, \cite[Proposition~5.45]{BM} gives completeness of the adic
filtration, namely $\operatorname*{holim}_q F^qX=0$.
The cotangent Hurewicz theorem for $k\to X$ shows that
$\gr^1X\simeq k\dtimes_X L_{X/k}$ is $2$-connective
\cite[Proposition~3.2.16]{Lurie}. Double d\'ecalage then gives
\[
 \gr^qX\simeq\LSym_k^q(\gr^1X)
 \simeq\bigl(\mathbb L\Gamma_k^q((\gr^1X)[-2])\bigr)[2q]
 \in\Mod_{k,\ge2q}
\]
by \cite[Lemma~3.1.2]{Lurie}. Finally, apply Lemma~\ref{lem:complete-filtration}.
\end{proof}

We can now lift to ordinary homotopy and identify the image in the
free target.

\begin{theorem}\label{thm:cycle-detection}
Suppose that $\fd_X L_{X/k}=d$ is finite and $d\ge3$. There exists
an augmented map $f:X\to S(d)$ such that for every $s\ge1$ with
$2p^s>d-3$ there is a class $z_s\in\pi_{d+2p^s-2}X$ satisfying
\begin{equation}\label{eq:detected-cycle}
 f_*(z_s)=b_{p,s}(\iota_d)\ne0 .
\end{equation}
\end{theorem}

\begin{proof}
By \eqref{eq:flat-dimension-field}, the top nonzero homotopy group
of $\gr^1X$ is in degree $d$. Choose $0\ne y\in\pi_d\gr^1X$ and take $f$ as in
Lemma~\ref{lem:detecting-map}. Set $t=d+2p^s-2$.
Proposition~\ref{prop:iterated-gap} lifts $b_{p,s}(y)$ to
$\pi_t(F^{p^s}X/F^{2p^s}X)$. The obstruction to lifting further to
$\pi_tF^{p^s}X$ lies in $\pi_{t-1}F^{2p^s}X$. The source
connectivity estimate \eqref{eq:adic-connectivity} kills this
obstruction, since $t-1=d+2p^s-3<4p^s$. Choose a lift
$\widehat z_s$ and let $z_s$ be its image in $\pi_t X$.

To identify $f_*(z_s)$, note that by naturality the image of $\widehat z_s$
in the weight-$p^s$ component
of $\pi_t F^{p^s}S(d)$ is $b_{p,s}(\iota_d)$. All higher components vanish
by the target connectivity estimate \eqref{eq:sphere-connectivity},
since $t<d+2(p^s+1)-2$. The split
filtration gives the equality in \eqref{eq:detected-cycle}, and
Lemma~\ref{lem:free-words} gives its nonvanishing.
\end{proof}

The detected class supports nonzero suspended operation strings of
every length. A uniform bound on the length of such strings therefore
forces the cotangent fiber to be concentrated in degrees at most two.

\begin{corollary}\label{thm:nilpotence-criterion}
Let $X$ be an augmented $1$-connected simplicial commutative
$k$-algebra with each $\pi_iX$ finite-dimensional over $k$.
Suppose $L_{X/k}$ has finite flat dimension over $X$. If there exists
$r\ge1$ such that every composable string of $r$ operations $\nu_a$,
each applied in its suspended range, vanishes on $\pi_{>0}X$, then
$\fd_X L_{X/k}\le2$.
\end{corollary}

\begin{proof}
If $d=\fd_X L_{X/k}\ge3$, choose $s$ and $z_s$ as in
Theorem~\ref{thm:cycle-detection}. For every $\ell\ge1$, naturality gives
\[
 f_*\bigl(\nu_{p^{s+\ell-1}}\cdots\nu_{p^s}z_s\bigr)
 =b_{p,s+\ell}(\iota_d)\ne0 .
\]
At step $j$, for $s\le j<s+\ell$, the input degree is
$d+2p^j-2>2p^j$, so all the operations in the string are suspended.
Taking $\ell=r$ contradicts the assumed uniform vanishing.
\end{proof}

\section{The boundary filtration}\label{sec:quotient}

The previous section used the adic filtration of the augmentation to detect
nonzero strings of operations. We now construct a different filtration that
forces sufficiently long strings of suspended operations to vanish.

Let $D$ be an augmented simplicial commutative $k$-algebra with
$\pi_0D=k$ and $\pi_1D=0$. Form the relative adic filtration of its
unit map, $K^\bullet=\adic(k\to D)$
\cite[Definition~2.5]{BMN}. Here we have $K^0=k$ and $\gr^0K\simeq D$ as algebras.
This puts operations on $\pi_*D$ in filtration index zero, where we
will apply Lemma~\ref{lem:relative}. The fiber sequence of $k$-modules
\[
 K^1\longrightarrow k\longrightarrow D
\]
gives an isomorphism
$\partial\colon\pi_mD\xrightarrow{\sim}\pi_{m-1}K^1$ for $m\geq2$.

\begin{definition}\label{def:boundary-filtration}
For $m\geq2$ and $q\geq1$, put
\begin{equation}\label{eq:boundary-filtration}
 G^q\pi_mD
 =\partial^{-1}\!\left(\operatorname{im}\bigl(
 \pi_{m-1}K^q\longrightarrow\pi_{m-1}K^1\bigr)\right).
\end{equation}
\end{definition}

Thus $x\in G^q\pi_mD$ precisely when its boundary lifts to $K^q$.
In particular $G^1\pi_mD=\pi_mD$ and $G^{q+1}\pi_mD\subseteq
G^q\pi_mD$. Write $k/K^q$ for the cofiber of $K^q\to k$ in
$\Mod_k$. Naturality of the connecting isomorphisms for $K^q\to k$
and $K^1\to k$ gives the equivalent description
\begin{equation}\label{eq:boundary-lift}
 G^q\pi_mD
 =\operatorname{im}\bigl(\pi_m(k/K^q)\longrightarrow\pi_mD\bigr).
\end{equation}

\begin{remark}\label{rem:geometric-boundary}
The unit defines a derived closed immersion
$Z=\Spec D\hookrightarrow\Spec k$, since it induces a surjection
on $\pi_0$. Its classical truncation is the identity of $\Spec k$.
Let $Z^{(q-1)}$ be its $(q-1)$-st derived infinitesimal neighborhood,
with $Z^{(0)}=Z$.
The Rees construction identifies $K^q$ with the degree-$q$ part of
the extended Rees algebra of this immersion
\cite[\S3.2 and Lemma~4.3.1]{hekking}. Its cofiber in $k$ is therefore
the underlying module of the functions on $Z^{(q-1)}$
\cite[Theorem~3.41(ii)]{HKR}.
Consequently,
\[
 G^q\pi_mD
 =\operatorname{im}\bigl(\pi_m\mathcal O(Z^{(q-1)})
                  \longrightarrow\pi_mD\bigr)
 \qquad(q\geq1,\ m\geq2).
\]
Thus $x\in G^q\pi_mD$ exactly when $x$ extends to a class in
$\pi_m\mathcal O(Z^{(q-1)})$. All these neighborhoods have the
same underlying classical point. Theorem~\ref{thm:boundary-nilpotence}
says that if $x$ extends across $Z^{(q-1)}$, then its suspended
operation $\nu_a(x)$ extends across $Z^{(pq-1)}$.
\end{remark}

To bound the filtration in all degrees at once,
suppose given a map of simplicial commutative rings $R\to S$ and a
map $R\to k$. Assume that $S$ is almost perfect over $R$ and that
$J=\operatorname{fib}(R\to S)$ has homotopy concentrated in $[1,h]$
for an integer $h\ge0$.
Set $J^\bullet=\adic(R\to S)$, so that $J^0=R$ and $J^1=J$.

The upper bound on $J$ will be used together with increasing
connectivity of the stages $J^q$. Orthogonality then makes the maps
$J^q\to J$ null for $q>h$. After base change to $k$, these
nullhomotopies will give a uniform cutoff for the boundary filtration
of $k\dtimes_R S$, even when this algebra has unbounded homotopy.

\begin{lemma}\label{lem:relative-adic-bound}
For each $q\geq1$, the $R$-module $J^q$ is $q$-connective.
In particular,
\begin{equation}\label{eq:relative-adic-bound}
 J^q\longrightarrow J^1 \text{ is null in }\Mod_R
 \qquad(q>h).
\end{equation}
\end{lemma}

\begin{proof}
The map $\pi_0R\to\pi_0S$ is an isomorphism. Its kernel is zero,
so $R$ is complete at this ideal, and almost perfectness of $S$ over $R$
implies that $J^\bullet$ is complete \cite[Proposition~2.18]{BMN}.
Its stages are connective and
$\gr^qJ\simeq\LSym_S^q(L_{S/R}[-1])$
\cite[Lemma~2.12]{BMN}. Since $J=\operatorname{fib}(R\to S)$ is
$1$-connective, the relative cotangent Hurewicz theorem implies that
$L_{S/R}$ is $2$-connective \cite[Proposition~3.2.16]{Lurie}.
By d\'ecalage,
\[
 \gr^qJ\simeq\bigl(\mathbb L\Lambda_S^q(L_{S/R}[-2])\bigr)[q]
 \in\Mod_{R,\geq q}.
\]
Lemma~\ref{lem:complete-filtration} now gives
$J^q\in\Mod_{R,\geq q}$. For $q>h$, orthogonality to
$J^1\in\Mod_{R,\leq h}$ gives \eqref{eq:relative-adic-bound}.
\end{proof}

\begin{lemma}\label{lem:base-changed-filtration}\label{lem:boundary-filtration}
Set $D=k\otimes_R^{\mathbb L} S$. Then $D$ is an augmented
$1$-connected simplicial $k$-algebra, and each $\pi_iD$ is
finite-dimensional over $k$. Its unit-adic filtration has stages
$K^q\simeq k\otimes_R^{\mathbb L}J^q$, and
\begin{equation}\label{eq:bounded-image-filtration}
 G^q\pi_mD=0\qquad(m\geq2,\ q>h).
\end{equation}
\end{lemma}

\begin{proof}
The map $R\to k$ factors through $\pi_0R\simeq\pi_0S$, giving
$S\to k$ and hence the augmentation $D\to k$. The fiber of $k\to D$
is $k\otimes_R^{\mathbb L}J$, which is $1$-connective. Thus
$\pi_0D=k$ and $\pi_1D=0$. Scalar extension preserves almost
perfect modules, so the homotopy groups of $D$ are finite-dimensional.

In the category of arrows, $(k\to D)$ is the pushout of
$(R\to S)\leftarrow(R\to R)\to(k\to k)$. The adic functor preserves
pushouts by its defining adjunction. The identity arrows have
trivial adic filtrations, so the resulting filtered pushout is
scalar extension from $R$ to $k$. Since these scalar algebras are
concentrated in filtration degree zero, the relative tensor product
is computed stagewise \cite[Remark~A.12]{BMN}. Thus
$K^q\simeq k\otimes_R^{\mathbb L}J^q$, compatibly with the
transition maps. In particular $K^q\to K^1$ is null for $q>h$
by \eqref{eq:relative-adic-bound}. Now apply
\eqref{eq:boundary-filtration}.
\end{proof}

\begin{theorem}\label{thm:boundary-nilpotence}
For $q,a\geq1$, $m>2a$, and $N=m+2a(p-1)$, the boundary filtration
satisfies
\begin{equation}\label{eq:boundary-operation}
 \nu_a(G^q\pi_mD)\subseteq G^{pq}\pi_ND.
\end{equation}
If $D=k\otimes_R^{\mathbb L}S$ comes from the bounded-fiber
presentation above, every composable string of $r$ operations
$\nu_a$, each applied in its suspended range, vanishes on the
positive homotopy of $D$ as soon as $p^r>h$.
\end{theorem}

\begin{proof}
By \eqref{eq:boundary-lift}, a class $x\in G^q\pi_mD$ lifts to
$\pi_m(K^0/K^q)$. Apply Lemma~\ref{lem:relative} to $K^\bullet$
at index zero with gap $q$ and weight $p$. Since $\nu_a$ is
suspended on degree $m$, the lemma lifts $\nu_a(x)$ to
$\pi_N(K^0/K^{pq})$. The operation on $\gr^0K\simeq D$ is the
usual $\nu_a$ on degree $m$.
Equation~\eqref{eq:boundary-lift} now gives \eqref{eq:boundary-operation}.
Starting in $G^1$, a string of length $r$ takes values in $G^{p^r}$,
which vanishes for $p^r>h$ by \eqref{eq:bounded-image-filtration}.
\end{proof}

\begin{corollary}\label{cor:relative-cotangent-flat-dimension}
Under the bounded-fiber hypotheses, finite flat dimension of
$L_{D/k}$ over $D$ implies $\fd_D L_{D/k}\leq2$.
Equivalently, if $k\otimes_S^{\mathbb L}L_{S/R}$ has finite flat
dimension over $k$, its flat dimension is at most $2$.
\end{corollary}

\begin{proof}
Lemma~\ref{lem:base-changed-filtration} and
Theorem~\ref{thm:boundary-nilpotence} give the hypotheses of
Corollary~\ref{thm:nilpotence-criterion} for any $r$ with $p^r>h$.
This proves the bound. Cotangent base change identifies
$k\otimes_D^{\mathbb L}L_{D/k}\simeq k\otimes_S^{\mathbb L}L_{S/R}$.
Since $\pi_0D=k$, \eqref{eq:flat-dimension-field} identifies the
flat dimension of $L_{D/k}$ with that of this fiber.
\end{proof}

\section{Quillen's conjecture for Noetherian rings}\label{sec:quillen}

\subsection{The difference algebra of a local surjection}

Let $(A,\mathfrak m,k)$ be a Noetherian local ring of residue
characteristic $p>0$ and let $S=A/I$, with $I\subseteq\mathfrak m$.
Choose generators $f_1,\ldots,f_c$ of $I$. We can then form the derived quotient
\[
 R=A//(f_1,\ldots,f_c)
\]
and measure the difference between the derived quotient $R$ and the classical
quotient by forming the \emph{difference algebra}
\[
    D = S \dtimes _R k.
\]
If the chosen generators form a regular
sequence, then $R\simeq S$ and $D\simeq k$.

The key property of the difference algebra is that it retains the
Andr\'e--Quillen homology of $S$ over $A$ starting in degree three.

\begin{proposition}\label{prop:difference-algebra}
The $R$-module $S$ is almost perfect, and
$J=\operatorname{fib}(R\to S)$ has homotopy concentrated in $[1,c]$.
There is a natural cofiber sequence
\begin{equation}\label{eq:difference-cotangent}
 k^c[1]\longrightarrow k\dtimes_S L_{S/A}
 \longrightarrow k\dtimes_D L_{D/k}.
\end{equation}
In particular, the second arrow gives natural isomorphisms
\begin{equation}\label{eq:difference-comparison}
 \pi_i(k\dtimes_S L_{S/A})
 \cong \pi_i(k\dtimes_D L_{D/k})\qquad(i\geq3).
\end{equation}
\end{proposition}

\begin{proof}
By the Koszul description of the derived quotient \cite[\S3.1]{tamme},
the homotopy groups of $R$ are finite $S$-modules concentrated in
degrees $0,\ldots,c$, with $\pi_0R=S$. Thus
$J=\operatorname{fib}(R\to S)$ is concentrated in $[1,c]$, and $R$
is coherent. Consequently, $S$ is almost perfect over $R$
\cite[Appendix~A]{BMN}.

Derived base change from the zero section $A[t_1,\ldots,t_c]\to A$
gives $L_{R/A}\simeq R^c[1]$. Transitivity for $A\to R\to S$,
followed by tensoring with $k$, gives the asserted cofiber sequence,
using the base-change equivalence
\[
 k\dtimes_S L_{S/R}\simeq k\dtimes_D L_{D/k}.
\]
The asserted isomorphisms follow from the long exact homotopy sequence.
\end{proof}

We now prove Quillen's conjecture for local surjections in positive
residue characteristic.

\begin{theorem}\label{thm:local-quillen}
For a quotient $S=A/I$ as above,
\[
 \fd_S L_{S/A}<\infty\quad\Longrightarrow\quad\fd_S L_{S/A}\leq2.
\]
Equivalently, if $\pi_i(k\dtimes_S L_{S/A})$ vanishes for all
sufficiently large $i$, it vanishes for all $i\geq3$.
\end{theorem}

\begin{proof}
The complex $L_{S/A}$ is connective and almost perfect. By
\eqref{eq:flat-dimension-field}, its flat dimension is the supremum
of the degrees $i$ with $\pi_i(k\dtimes_S L_{S/A})\ne0$.

Suppose that these groups vanish for $i\gg0$. By
\eqref{eq:difference-comparison}, the same holds for the cotangent
fiber of $D$. Proposition~\ref{prop:difference-algebra} verifies the
hypotheses of Corollary~\ref{cor:relative-cotangent-flat-dimension},
so the latter fiber vanishes in degrees $i\geq3$. A second
application of \eqref{eq:difference-comparison} gives the result.
\end{proof}

\subsection{Passage to arbitrary homomorphisms}

Cohen factorization gives a global consequence for arbitrary
homomorphisms with target of positive characteristic.

\begin{corollary}
\label{cor:quillen-finite-characteristic}
Let $A\to B$ be a homomorphism of Noetherian rings. If $B$ has
positive characteristic, then
\[
 \fd_B L_{B/A}<\infty\quad\Longrightarrow\quad\fd_B L_{B/A}\leq2.
\]
\end{corollary}

\begin{proof}
Let $\mathfrak q\in\Spec B$, set $\mathfrak p=\mathfrak q\cap A$,
and write $k=k(\mathfrak q)$. Choose a Cohen factorization
\[
 A_{\mathfrak p}\longrightarrow A'
 \twoheadrightarrow\widehat{B_{\mathfrak q}}
\]
with $A'$ a complete Noetherian local ring with residue field $k$
\cite[Theorem~(1.1)]{avramov-foxby-herzog}.
Localization and the Cohen-factorization comparison
\cite[Lemma~(1.7) and Remark~(1.9)]{avramov} give
\[
 \pi_i(k\dtimes_B L_{B/A})\cong
 \pi_i\bigl(k\dtimes_{\widehat{B_{\mathfrak q}}}
                 L_{\widehat{B_{\mathfrak q}}/A'}\bigr)
 \qquad(i\geq2).
\]
Whenever $k$ has positive characteristic and the left side vanishes
for $i\gg0$, Theorem~\ref{thm:local-quillen}, applied to
$A'\twoheadrightarrow\widehat{B_{\mathfrak q}}$, gives vanishing
for all $i\geq3$.

Under the hypotheses of the corollary, all residue fields have
positive characteristic, and finite flat dimension gives eventual
vanishing at every prime. Andr\'e's residue-field criterion
\cite[(S.29)]{andre} therefore gives vanishing in degrees at least
three with coefficients in every $B$-module, hence
$\fd_B L_{B/A}\leq2$.
\end{proof}

For local homomorphisms with perfect cotangent complex, the same
comparison gives a mixed-characteristic consequence.

\begin{corollary}
\label{cor:quillen-mixed-characteristic}
Let $A\to B$ be a local homomorphism of Noetherian local rings
whose target has residue field $k$ of positive characteristic.
If $L_{B/A}$ is perfect over $B$, then it has Tor-amplitude
contained in $[0,2]$.
\end{corollary}

\begin{proof}
Choose a Cohen factorization
\[
 A\longrightarrow A'\twoheadrightarrow\widehat B
\]
\cite[Theorem~(1.1)]{avramov-foxby-herzog}. Avramov's comparison
\cite[Lemma~(1.7)]{avramov} gives
\[
 \pi_i(k\dtimes_B L_{B/A})\cong
 \pi_i(k\dtimes_{\widehat B}L_{\widehat B/A'})
 \qquad(i\geq2).
\]
Perfection of $L_{B/A}$ gives eventual vanishing on the left.
Theorem~\ref{thm:local-quillen}, applied to
$A'\twoheadrightarrow\widehat B$, then gives vanishing on both
sides for $i\geq3$. Since $L_{B/A}$ is connective and perfect,
its Tor-amplitude is detected at $k$
\cite[Tag~068V]{stacks-project}, giving the bound $[0,2]$.
\end{proof}

For an essentially finite type local homomorphism as in
Corollary~\ref{cor:quillen-mixed-characteristic}, the cotangent
complex is almost perfect. Perfection is therefore equivalent to
finite flat dimension \cite[Tag~0658]{stacks-project}, so the
corollary proves the usual local formulation of Conjecture II in
positive residue characteristic.

\section{Maps of simplicial rings}\label{sec:non-discrete}

Following Turner \cite{TurnerNoetherian}, a simplicial
commutative ring $A$ has \emph{finite Noetherian homotopy} if
$\pi_0A$ is Noetherian and $\bigoplus_{i\geq0}\pi_iA$ is a finitely
generated $\pi_0A$-module.

Turner proposed that Quillen's conjecture should hold for simplicial algebras
with this finiteness property. We prove this is actually true for morphisms
surjective on $\pi_0$ by extending the construction of Section~\ref{sec:quillen}
allowing both source and target to be non-discrete, but with finite Noetherian
homotopy.

The additional step is to make the intermediate map surjective on
$\pi_1$. Together with an isomorphism on $\pi_0$, this ensures that
its fiber is $1$-connective, as required by the boundary-filtration
argument.

Throughout this section $A\to B$ is a map of simplicial commutative
rings such that both $A$ and $B$ have finite Noetherian homotopy and
$\pi_0A\to\pi_0B$ is surjective.

\begin{lemma}
\label{lem:derived-quotient-factorisation}
There is a factorization $A\to R\to B$ with $R$ of finite
Noetherian homotopy such that $\pi_0R\to\pi_0B$ is an isomorphism
and $\pi_1R\to\pi_1B$ is surjective. The $R$-module $B$ is
almost perfect, and there are integers $e,h\geq0$ such that
\[
 L_{R/A}\simeq R^e[1],\qquad
 \operatorname{fib}(R\to B)\in\Mod_{R,[1,h]}.
\]
\end{lemma}

\begin{proof}
We first arrange an isomorphism on $\pi_0$.
Choose generators $f_1,\ldots,f_c$ of $\ker(\pi_0A\to\pi_0B)$
and representatives in $A$. Nullhomotopies of their images in $B$
give a factorization through $R_0=A//(f_1,\ldots,f_c)$.
Its zeroth homotopy ring is $\pi_0B$. Its underlying module is an
iterated Koszul complex over $A$, so its homotopy is bounded and
finite over $\pi_0A$, hence over $\pi_0R_0$.

Next we make the map on $\pi_1$ surjective.
Taking the derived quotient of $R_0$ by zero attaches a cell in
degree one. Extending the map $R_0\to B$ over this cell amounts
to choosing a loop at zero in $B$, representing a class of $\pi_1B$.
Choose $\ell$ classes whose images generate the cokernel of
$\pi_1R_0\to\pi_1B$ over $\pi_0B$. Representatives of these
classes therefore specify a map to $B$ from
\[
 R=R_0//(\underbrace{0,\ldots,0}_{\ell})
 \simeq A//(f_1,\ldots,f_c,\underbrace{0,\ldots,0}_{\ell}).
\]
The Koszul complex on $\ell$ zeros yields an $R_0$-module equivalence
\[
 R\simeq\bigoplus_{j=0}^{\ell} R_0^{\binom{\ell}{j}}[j].
\]
Thus $R$ has finite Noetherian homotopy and $\pi_0R=\pi_0B$.
On $\pi_1$, the extra summand $(\pi_0B)^\ell$ maps to the chosen
generators, giving surjectivity. The long exact sequence and the
cotangent calculation of Proposition~\ref{prop:difference-algebra}
therefore give, for some $h\geq0$,
\[
 L_{R/A}\simeq R^{c+\ell}[1],\qquad
 \operatorname{fib}(R\to B)\in\Mod_{R,[1,h]}.
\]
The coherence of $R$ and finite homotopy of $B$ over $\pi_0R$ give
almost perfectness of $B$ over $R$ \cite[Appendix~A]{BMN}.
\end{proof}

If $B$ is discrete, no additional degree-one generators are needed:
we may take $\ell=0$ in
Lemma~\ref{lem:derived-quotient-factorisation}. When both $A$ and
$B$ are discrete, this is the construction of Section~\ref{sec:quillen}.

\begin{theorem}\label{thm:surjective-derived-flat-dimension}
Let $A\to B$ be a map of simplicial commutative rings such that
both $A$ and $B$ have finite Noetherian homotopy and
$\pi_0A\to\pi_0B$ is surjective. Suppose that $\pi_0B$ has
positive characteristic or that it is local with residue field
of positive characteristic. Then
\[
 \fd_B L_{B/A}<\infty\quad\Longrightarrow\quad\fd_B L_{B/A}\leq2.
\]
\end{theorem}

\begin{proof}
Assume that $\fd_B L_{B/A}<\infty$.
Since $A$ is coherent and the homotopy of $B$ is bounded and finite
over $\pi_0A$, $B$ is almost perfect as an $A$-module.
The surjection on $\pi_0$ and \cite[Proposition~2.14]{BMN} then give
almost perfectness of $L_{B/A}$ over $B$.
All discrete $B$-modules are $\pi_0B$-modules, so
\[
 \fd_B L_{B/A}=\fd_{\pi_0B}(\pi_0B\dtimes_B L_{B/A}).
\]
The $\pi_0B$-module $\pi_0B\dtimes_B L_{B/A}$ is connective and
almost perfect. By residue-field detection of Tor-amplitude
\cite[Tag~068V]{stacks-project}, it suffices to prove
\[
 \pi_i(k\dtimes_B L_{B/A})=0\qquad(i\geq3)
\]
at every residue field of $\pi_0B$ when $\pi_0B$ has positive
characteristic, and at its closed residue field when $\pi_0B$ is
local. In either case, $k$ has positive characteristic.

Fix such a field $k$. Choose $A\to R\to B$ as in
Lemma~\ref{lem:derived-quotient-factorisation} and put $D=k\dtimes_R B$.
Cotangent transitivity and base change give a cofiber sequence
\begin{equation}\label{eq:derived-difference-cotangent}
 k^e[1]\longrightarrow k\dtimes_B L_{B/A}
 \longrightarrow k\dtimes_D L_{D/k}.
\end{equation}
The last two terms have isomorphic homotopy in degrees $i\geq3$.
The middle term has bounded homotopy by assumption, so the right
term does as well. Lemma~\ref{lem:derived-quotient-factorisation}
verifies the bounded-fiber and almost-perfectness hypotheses of
Corollary~\ref{cor:relative-cotangent-flat-dimension} for $R\to B\to k$.
That corollary gives
\[
 \pi_i(k\dtimes_D L_{D/k})=0\qquad(i\geq3).
\]
The cofiber sequence gives the required vanishing for
$k\dtimes_B L_{B/A}$, and hence the flat-dimension bound.
\end{proof}

\begingroup
\raggedright

\endgroup


\begin{thebibliography}{10}

\bibitem{andre}
Michel Andr\'e.
\newblock {\em Homologie des alg\`ebres commutatives}, volume 206 of {\em Die
  Grundlehren der mathematischen Wissenschaften}.
\newblock Springer-Verlag, 1974.

\bibitem{avramov}
Luchezar~L. Avramov.
\newblock Locally complete intersection homomorphisms and a conjecture of
  {Q}uillen on the vanishing of cotangent homology.
\newblock {\em Ann. of Math. (2)}, 150(2):455--487, 1999.

\bibitem{avramov-foxby-herzog}
Luchezar~L. Avramov, Hans-Bj{\o}rn Foxby, and Bernd Herzog.
\newblock Structure of local homomorphisms.
\newblock {\em J. Algebra}, 164(1):124--145, 1994.

\bibitem{avramoviyengar}
Luchezar~L. Avramov and Srikanth Iyengar.
\newblock Andr\'e-{Q}uillen homology of algebra retracts.
\newblock {\em Ann. Sci. \'Ecole Norm. Sup. (4)}, 36(3):431--462, 2003.

\bibitem{BM}
D.~Lukas~B. Brantner and Akhil Mathew.
\newblock Deformation theory and partition {Lie} algebras.
\newblock {\em Acta Math.}, 235(1):1--148, 2025.
\newblock arXiv:1904.07352v5.

\bibitem{BMN}
Lukas Brantner, Kirill Magidson, and Joost Nuiten.
\newblock Formal integration of derived foliations, 2025.
\newblock arXiv:2502.05257v2.

\bibitem{briggs-iyengar}
Benjamin Briggs and Srikanth~B. Iyengar.
\newblock Rigidity properties of the cotangent complex.
\newblock {\em J. Amer. Math. Soc.}, 36(1):291--310, 2023.

\bibitem{FN}
Jiaqi Fu and Joost Nuiten.
\newblock Homotopy operations on higher {Lie} algebras, 2026.
\newblock arXiv:2610.36942v1.

\bibitem{hekking}
Zachary Gardner and Jeroen Hekking.
\newblock A note on ideals in derived geometries, 2025.
\newblock arXiv:2511.20331v1.

\bibitem{HKR}
Jeroen Hekking, Adeel~A. Khan, and David Rydh.
\newblock Deformation to the normal bundle and blow-ups via derived {Weil}
  restrictions, 2025.
\newblock arXiv:2511.19412.

\bibitem{iyengarintro}
Srikanth Iyengar.
\newblock Andr{\'e}-{Q}uillen homology of commutative algebras.
\newblock In {\em Interactions between homotopy theory and algebra}, volume 436
  of {\em Contemporary Mathematics}, pages 203--234. American Mathematical
  Society, Providence, RI, 2007.

\bibitem{kst}
Shane Kelly, Shuji Saito, and Georg Tamme.
\newblock On pro-cdh descent on derived schemes.
\newblock {\em Geom. Topol.}, 30(1):337--372, 2026.

\bibitem{tamme}
Moritz Kerz, Florian Strunk, and Georg Tamme.
\newblock Algebraic {K}-theory and descent for blow-ups.
\newblock {\em Invent. Math.}, 211(2):523--577, 2018.

\bibitem{Liu}
Fei Liu.
\newblock {Quillen}'s finiteness conjecture in positive residual
  characteristic.
\newblock Zenodo, 2026.
\newblock
  \href{https://doi.org/10.5281/zenodo.23466728}{doi:10.5281/zenodo.23466728}.

\bibitem{Lurie}
Jacob Lurie.
\newblock Derived algebraic geometry.
\newblock Manuscript. \url{https://www.math.ias.edu/~lurie/papers/DAG.pdf}.

\bibitem{quillen-proc}
Daniel Quillen.
\newblock On the (co-)homology of commutative rings.
\newblock In {\em Applications of Categorical Algebra (Proc. Sympos. Pure
  Math., Vol. XVII, New York, 1968)}, pages 65--87. Amer. Math. Soc., 1970.

\bibitem{stacks-project}
{The Stacks Project Authors}.
\newblock The {Stacks Project}.
\newblock \url{https://stacks.math.columbia.edu}, 2026.

\bibitem{HAGII}
Bertrand To{\"e}n and Gabriele Vezzosi.
\newblock Homotopical algebraic geometry {II}: Geometric stacks and
  applications.
\newblock {\em Mem. Amer. Math. Soc.}, 193(902):x+224, 2008.

\bibitem{turner-inventiones}
James~M. Turner.
\newblock On simplicial commutative algebras with vanishing
  {Andr{\'e}-{Q}uillen} homology.
\newblock {\em Invent. Math.}, 142(3):547--558, 2000.

\bibitem{TurnerNoetherian}
James~M. Turner.
\newblock On simplicial commutative algebras with {Noetherian} homotopy.
\newblock {\em J. Pure Appl. Algebra}, 174(2):207--218, 2002.

\bibitem{turner}
James~M. Turner.
\newblock On the rigidity of the cotangent complex at the prime 2.
\newblock {\em J. Pure Appl. Algebra}, 213(7):1224--1238, 2009.

\end{thebibliography}
\end{document}